\documentclass{amsart}
\usepackage{amsfonts}
\usepackage{amsmath,amscd}
\usepackage{amssymb}
\usepackage{amsthm}
\usepackage{newlfont}

\long\def\alert#1{\parindent2em\smallskip\hbox to\hsize%
{\hskip\parindent\vrule%
\vbox{\advance\hsize-2\parindent\hrule\smallskip\parindent.4\parindent%
\narrower\noindent#1\smallskip\hrule}\vrule\hfill}\smallskip\parindent0pt}
 \newtheorem{thm}{Theorem}[section]
 \newtheorem{cor}[thm]{Corollary}
 \newtheorem{lem}[thm]{Lemma}
 \newtheorem{prop}[thm]{Proposition}
 \theoremstyle{definition}
 \newtheorem{defn}[thm]{Definition}
 \theoremstyle{remark}

 \numberwithin{equation}{section}

\begin{document}


%
\title[ Lie algebras with the derived subalgebra of dimension two ]
{ Capable Lie algebras with the derived subalgebra of dimension two over an arbitrary field }
\author[P. Niroomand]{Peyman Niroomand}
\author[F. Johari]{Farangis Johari}
\author[M. Parvizi]{Mohsen Parvizi}
\address{School of Mathematics and Computer Science\\
Damghan University, Damghan, Iran}
\email{p$\_$niroomand@yahoo.com, niroomand@du.ac.ir}
\address{Department of Pure Mathematics\\
Ferdowsi University of Mashhad, Mashhad, Iran}
\email{farangis.johari@mail.um.ac.ir, farangisjohary@yahoo.com}

\address{Department of Pure Mathematics\\
Ferdowsi University of Mashhad, Mashhad, Iran}
\email{parvizi@math.um.ac.ir}



\date{\today}

\begin{abstract}
In this paper, we classify all capable nilpotent Lie algebras with the derived subalgebra of dimension 2 over an arbitrary field.  Moreover, the explicit structure of such Lie algebras of class 3 is given.
 \end{abstract}

\maketitle

\section{Introduction and  Motivation}
The concept of capability was introduced by P. Hall in \cite{hall}. Recall that  a group $ G $ is called  capable if there exists some group $ E $ such that $ G\cong E/Z(E), $ where $Z(E)  $ denotes the center of $ E.$
There are some fundamental known results concerning capability of $p$-groups. For instance, in \cite[Corollary 4.16]{3}, it is shown the only capable extra-special $p$-groups (the $p$-groups with $Z(G)=G'$ and $|G'|=p$) are those of order $p^3$ and exponent $p$.
In the case that $G'=Z(G)$ and $Z(G)$ is elementary abelian $p$-group of rank $2$, Heineken in \cite{hei}
proved that the capable ones has order at most $p^7$.\newline
   By some results due to Lazard, we may associate a $p$-group to a Lie algebra. Therefore some results of  Lie algebras and $p$-groups
   have similarities in the structure. But in this way not every thing are the same and there are differences between groups and Lie algebras, so that most of time the proofs are different. Similar to the concept of the capability for groups, a Lie algebra is called capable provided that $L\cong H/Z (H) $ for a Lie algebra $H.$
Beyl et al. in \cite{3} introduced the epicenter $Z^*(G)  $ of a group $ G $ that plays an important role in the capability of $G.$ The analogous notion of the epicenter, $Z^*(L)  $ for a   Lie algebra $L $ was defined in \cite{alam}. It is shown that $ L $ is capable if and only if $Z^*(L)=0.$\newline
Another notion having relation to the capability is the concept of exterior square of Lie algebras, $ L\wedge L, $ which was introduced in \cite{el}.
Our approach is on the concept of the exterior center $ Z^{\wedge}(L),$ the set of all elements $l$
of $L$ for which $l \wedge l' = 0_{L\wedge L } $ for all $l'\in L.$ Niroomand et al. in \cite{ni3} showed $Z^{\wedge}(L)=Z^*(L) $ for any finite dimensional Lie algebra $ L.$ In \cite{ni3}, the last two authors obtained the structure of a capable nilpotent Lie algebra $L$ when $\dim L^2\leq 1$. It  developed the result of \cite[Corollary 4.16]{3} for groups to the class of Lie algebras.\\
 Recall that from \cite{ni4}, a Lie algebra $ H $ is called generalized Heisenberg of rank $ n $ if $ H^2=Z(H) $ and $ \dim H^2=n $. If $ n=1,$ then $ H $ is a Heisenberg Lie algebra that is more well-known. Such algebras are odd dimension and have the  presentation $ H(m)\cong \langle a_1,b_1,\ldots, a_m,b_m,z\big{|}[a_l,b_l]=z,1\leq
l\leq m\rangle.$
\\
Recently, Niroomand et al.  in \cite{ni4} proved the capable generalized Heisenberg Lie algebras of rank $ 2 $ have dimension at most $ 7$ over a filed of characteristic not equal to $ 2.$
 They developed  the result of Heineken \cite{hei} for groups to the area of Lie algebras.
They also characterized the structure of all capable nilpotent Lie algebras of class two when $ \dim L^2=2. $
  In  virtue of the recent results in \cite{ni4}, in this paper we intend to classify the structure of all capable nilpotent Lie algebras of class two with the derived subalgebra of dimension $ 2$ over an arbitrary filed.
 Furthermore, we determine the structure of all nilpotent Lie algebras of class $ 3 $ with the derived subalgebra of dimension $ 2$ and then we specify which one of them are capable.

\section{ Preliminaries}
All Lie algebras in this paper are finite dimensional.
The Schur multiplier of a Lie algebra $ L,$ $ \mathcal{M}(L),$ is defined as $ \mathcal{M}(L)\cong R\cap F^2/[R,F]$ where $L\cong F/R $
and $ F $ is a free Lie algebra. It can be shown that the Lie algebra $\mathcal{M}(L)$ is abelian and independent
of the choice of the free Lie algebra $F$ (see  \cite{ba1,2b,bos,ha, mon, ni, ni1,ni2, ni3} for more information on this topics).\\
Throughout the paper  $ \otimes $  used to denote  the operator of usual tensor product of algebras. For a Lie algebra $L,$ we denote
 the factor Lie algebra $L/L^2$ by $L^{(ab)}.$ Also we denote an abelian Lie algebra of dimension $n$ by $A(n).$\\
The following proposition plays a key role in detecting the capability of Lie algebras.
\begin{prop}\label{dd}
Let $ L $ be a finite dimensional Lie algebra with a central ideal $ I.$ Then
\begin{itemize}
\item[$(i)  $]$ \dim (\mathcal{M}(L))\geq \dim( \mathcal{M}(L/I))-\dim (L^{2}\cap I),$
\item[$(ii)  $]$ \dim (\mathcal{M}(L))=\dim( \mathcal{M}(L/I))-\dim (L^{2}\cap I) $ if and only if $ I\subseteq Z^{*}(L).$
\end{itemize}
\end{prop}
\begin{proof}
 The result follows from  \cite[Proposition 4.1(iii) and Theorem 4.4 ]{alam}.
\end{proof}

The following lemma from \cite{alam} is a useful instrument in the next investigations.
 \begin{lem}\label{cor1}
Let $I$ be an ideal of a Lie algebras $L$ such that $ L/I$ is
capable. Then $Z^{*}(L)\subseteq I.$
\end{lem}
The next corollary shows that the epicenter of a Lie algebra is contained in its derived subalgebra.
\begin{cor}\label{lll}
Let $ L $ be a finite dimensional non-abelian nilpotent Lie algebra. Then  $ Z^{*}(L) \subseteq L^2.$
\end{cor}
\begin{proof}
Since $ \dim  L/L^2\geq 2, $ $ L/L^2 $ is capable, using \cite[Theorem 3.3]{ni3}. Therefore the result follows by  Lemma \ref{cor1}.
\end{proof}

We need the notion of a central product of Lie algebras, it is defined as follows.

\begin{defn}\label{cent}
The Lie algebra $L$ is a central product of $A$ and $B,$ if $ L=A+B,$ where  $A$ and $B$ are ideals of $ L $ such that $ [A,B]=0$ and $A\cap B\subseteq Z(L).$ We denote the central product of two Lie algebras $A$ and $B$ by $A\dotplus B.$
\end{defn}

The   Heisenberg Lie algebra can be presented in terms of central products.

\begin{lem}\cite[Lemma 3.3]{pair}\label{fr}
Let $ L $ be a Heisenberg Lie algebra of dimension $2m+1.$ Then $ L $ is central products of its ideals  $B_j$ for all $1\leq j\leq m$ such that $B_j$ is  the Heisenberg Lie algebra of dimension $3.$
\end{lem}
 It is not  an easy matter to determine  the capability
of a central product in general. The next result gives the answer in a particular case.
\begin{prop}\label{Hi}\cite[Proposition 2.2]{ni4}
Let $L$ be a  Lie algebra such that $L=A\dotplus B$ with  $ A^2\cap B^2\neq 0.$ Then $ A^2\cap B^2\subseteq Z^{\wedge}(L)$ and so $ L $ is non-capable.
\end{prop}

Following \cite[Propositions 2.4 and 2.6]{ni4}, for determining the capable nilpotent Lie algebras of class $ 2 $ with the derived subalgebra of dimension $2,$ it is enough to consider a generalized Heisenberg Lie algebra $ H $  when
 \[5\leq \dim  H\leq 7.\]
 Throughout the paper  multiplication tables with respect to a fixed basis
with trivial products of the form $[x, y] = 0$ omitted, when $ x,y $ belongs to the Lie algebra.\\ Nilpotent
Lie algebras of dimension at most $5  $ can be described uniformly over all fields in \cite{cic,Gr2,Gr}. For the dimension $ 6,7 $ over an algebraic closed filed the structure are known in \cite{cic,Gr2,Gr}.
Using notation and terminology of \cite{cic,Gr2,Gr},  we list the generalized Heisenberg Lie algebras of rank two of dimension at most $6$ and $7$ over field of characteristic 2 and characteristic different 2, respectively.
Recall that in a field $\mathbb{F}$ of characteristic $ 2,$ $ \omega $
denotes a fixed element from $\mathbb{F} \setminus \{x^2+x|x\in \mathbb{F} \}.$
 \begin{thm}\label{11233}
Let $H$ be a generalized Heisenberg Lie algebra of rank $2.$ Then
\begin{itemize}
\item[$(i)$]over a field $\mathbb{F}$ of characteristic $ 2,$ the list of the isomorphism types of generalized Heisenberg Lie algebras of dimension at most  $ 6 $ are the following:
$L_{5,8}=\langle  x_1,\ldots, x_5\big{|}[x_1, x_2] = x_4, [x_1, x_3] = x_5\rangle$ and
$L_{6,7}^{(2)}(\eta) = \langle x_1,\ldots, x_6 | [x_1, x_2] = x_5, [x_1, x_3] = x_6, [x_2, x_4] = \eta x_6, [x_3, x_4] = x_5 + x_6\rangle $
where $\eta \in \{0, \omega\},$
\item[$(ii)$]over a field $\mathbb{F}$ of  characteristic different from $ 2,$  the list of the isomorphism types of generalized Heisenberg Lie algebras of dimension at most  $ 7 $ are
\[L_{5,8}=\langle  x_1,\ldots, x_5\big{|}[x_1, x_2] = x_4, [x_1, x_3] = x_5\rangle,\]
$ L_{6,22}(\epsilon)=\langle   x_1,\ldots, x_6\big{|}[x_1, x_2] = x_5= [x_3, x_4], [x_1, x_3] = x_6,[x_2, x_4] = \epsilon x_6 \rangle,$ where $\epsilon \in \mathbb{F}/(\overset{*}{\sim})$ and  $~\text{char}~ \mathbb{F} \neq 2,$
 \[L_1=27A=  \langle x_1,\ldots, x_7\big{|}[x_1, x_2] = x_6=[x_3, x_4],[x_1, x_5] = x_7= [x_2, x_3]\rangle,\] \[L_2=27B=  \langle  x_1,\ldots, x_7\big{|}[x_1, x_2] = x_6, [x_1, x_4] = x_7=[x_3, x_5] \rangle. \]
 \end{itemize}
\end{thm}

\section{capable  generalized Heisenberg Lie algebras of rank two}
Here, we are going to determine the structures of capable  generalized Heisenberg Lie algebras of rank two. By \cite[Proposition 2.6]{ni4}, generalized Heisenberg Lie algebras with
the derived subalgebra of dimension $ 2 $ are capable if their dimension lies between $ 5 $ and $ 7.$ According to  Theorem \ref{11233},
we have the presentation of all capable generalized Heisenberg  Lie algebras of rank two over a filed $F$ with $\text{char}~ \mathbb{F}\neq 2.$ But when $\text{char}~ \mathbb{F}= 2$ the structure of them is unknown. Therefore, at first we intend to find the structure of them on an arbitrary filed and then
we determine which ones are capable.

\begin{thm}\label{klj1}
Let $ L $ be a $7$-dimensional generalized Heisenberg Lie algebra over a filed $F$ with $\text{char}~ \mathbb{F}= 2$ of rank two. Then  $L\cong  \langle x_1,\ldots, x_7\big{|}[x_1, x_2] = x_6=[x_3, x_4],[x_1, x_5] = x_7= [x_2, x_3]\rangle$ or $L\cong \langle  x_1,\ldots, x_7\big{|}[x_1, x_2] = x_6, [x_1, x_4] = x_7=[x_3, x_5] \rangle. $
\end{thm}
\begin{proof}
 Let $ L^2=\langle z_1\rangle \oplus \langle z_2\rangle.$ By \cite[Theorem 3.6]{ni3}, we have $ L/\langle z_2\rangle\cong H(2)\oplus A(1)$ or $ L/\langle z_2\rangle\cong H(1)\oplus A(3).$ First suppose that  $ L/\langle z_2\rangle\cong H(2)\oplus A(1).$
 There exist two  ideals $ I_1/\langle z_2\rangle$ and $ I_2/\langle z_2\rangle$  of $ L/\langle z_2\rangle$ such that
\[  I_1/\langle z_2\rangle\cong H(2) ~ \text{and}~
 I_2/\langle z_2\rangle\cong A(1).\]
 Clearly, $ L=I_1+I_2, I_1\cap I_2= \langle z_2\rangle,$  $ [I_1,I_2]\subseteq  \langle z_2\rangle $ and $ [L,I_2]\subseteq  \langle z_2\rangle. $
 Thus $ I_1/\langle z_2\rangle=\langle x_1+\langle z_2\rangle,y_1+\langle z_2\rangle,x_2+\langle z_2\rangle,y_2+\langle z_2\rangle,z_1+\langle z_2\rangle|[x_1,y_1]+\langle z_2\rangle=[x_2,y_2]+\langle z_2\rangle=z_1+\langle z_2\rangle\rangle  $ and $ I_2=\langle q\rangle\oplus \langle z_2\rangle \cong A(2)$  for some $ q\in L.$ Hence the set $ \{x_1,y_1,x_2,y_2,z_1,z_2,q\} $ is a basis of $ L$  and
 \begin{align*}
 &[x_1,y_1]=z_1+\alpha_1 z_2,[x_2,y_2]=z_1+\alpha_2 z_2,\\
& [x_1,y_2]=\alpha_3 z_2,[x_1,x_2]=\alpha_4 z_2,\\
 & [x_2,y_1]=\alpha_5 z_2,[y_1,y_2]=\alpha_6 z_2,\\
 &  [x_1,q]=\alpha_7 z_2,[y_1,q]=\alpha_8 z_2,\\
  & [x_2,q]=\alpha_9 z_2,[y_2,q]=\alpha_{10} z_2.
 \end{align*}
 By changing variable, we assume that $\alpha_1=\alpha_2=0.  $
  Since $ q\notin L^2=Z(L)=\langle z_1\rangle \oplus \langle z_2\rangle,$ $  q$ is not central. Thus $ [L,I_2]=[\langle q\rangle,I_2]= \langle z_2\rangle.$ Without loss of
generality, assume that $\alpha_7 \neq 0.  $  By \cite[Theorem 3.6]{ni3}, we have $ L/\langle z_1\rangle\cong H(2)\oplus A(1)$ or $ L/\langle z_1\rangle\cong H(1)\oplus A(3).$ First suppose that  $ L/\langle z_1\rangle\cong H(2)\oplus A(1).$
 There exist two  ideals $ I_3/\langle z_1\rangle$ and $ I_4/\langle z_1\rangle$  of $ L/\langle z_1\rangle$ such that
\[  I_3/\langle z_1\rangle\cong H(2) ~ \text{and}~
 I_4/\langle z_1\rangle\cong A(1).\]
 Clearly, $ L=I_3+I_4 $ and $ [I_3,I_4]\subseteq  \langle z_1\rangle.$  \\ We claim that $ q,x_1\in I_3 $ and $[x_2,q]=[y_1,q]=[y_2,q]=[x_1,x_2]=[x_1,y_2]=0.$ Let  $ a+\langle z_1\rangle \in I_3 /\langle z_1\rangle $ and $ I_4 /\langle z_1\rangle =\langle b+\langle z_1\rangle \rangle$ such that $ q+ \langle z_1\rangle=( a+\langle z_1\rangle )+(\alpha b+\langle z_1\rangle). $ If $  a+\langle z_1\rangle=0,$ then since $ [L,I_2]=[\langle q\rangle,I_2]= \langle z_2\rangle,$ we have $ [\langle q\rangle,L]\in \langle z_1\rangle \cap   \langle z_2\rangle=0. $ So $ q\in Z(L)=L^2= \langle  z_1\rangle \oplus \langle  z_2\rangle.$ It is a contradiction. Thus $ q-a-\alpha b\in  \langle z_1\rangle$ and $ a+\langle z_1\rangle \neq 0. $ We have
 \begin{align*}
  &a=\eta_1 x_1+\eta_2 x_2+\eta_3 y_1+\eta_4 y_2+\eta_5 z_1+\eta_6 z_2+\eta_7 q, \\
   &\alpha b=\eta_1' x_1+\eta_2' x_2+\eta_3' y_1+\eta_4' y_2+\eta_5' z_1+\eta_6' z_2+\eta_7' q,
 \end{align*}
 and so \begin{align*}
    &q=a+\alpha b+\gamma z_1=\eta_1 x_1+\eta_2 x_2+\eta_3 y_1+\eta_4 y_2+\eta_5 z_1+\eta_6 z_2+\eta_7 q+\\&\eta_1' x_1+\eta_2' x_2+\eta_3' y_1+\eta_4' y_2+\eta_5' z_1+\eta_6' z_2+\eta_7' q+\gamma z_1.
    \end{align*}
 Since the set $ \{ x_1,y_1,x_2,y_2,z_1,z_2,q \} $ is linearly independent and $ L $ is a Lie algebra on filed over characteristic two, $  \eta_1=\eta_1',\eta_2=\eta_2',\eta_3=\eta_3',\eta_4=\eta_4',\eta_5=-\eta_5'-\gamma,\eta_6=\eta_6',\eta_7=1-\eta_7'.$ Thus \begin{align*}
 & q+\langle z_1\rangle =a+\langle z_1\rangle+\alpha b+\langle z_1\rangle =(\eta_1+\eta_1')x_1+(\eta_2+\eta_2') x_2+(\eta_3+\eta_3')y_1+\\&(\eta_4+\eta_4')y_2+(\eta_6+\eta_6')z_2+
 (\eta_7+\eta_7')q+\langle z_1\rangle.
 \end{align*}
 We conclude that $ q+\langle z_1\rangle= (\eta_7+\eta_7')q+\langle z_1\rangle$ so $ \eta_7\neq 0 $ or $ \eta_7'\neq 0. $ Since $ q+\langle z_1\rangle \notin I_4/\langle z_1  \rangle, $ we have $ \eta_7\neq 0 $ and $ \eta_7'= 0. $ Thus $ q+\langle z_1\rangle \in  I_3/\langle z_1  \rangle.  $ Hence $ q\in I_3. $
  Now, we prove that $x_1\in I_3.$ By contrary, assume that $x_1\notin I_3.$ We may assume that $x_1+\langle z_1\rangle=a_1+\langle z_1\rangle+\alpha'b+\langle z_1\rangle$, where $a_1+\langle z_1\rangle\in I_3/\langle z_1\rangle$ and $I_4/\langle z_1\rangle=\langle b+\langle z_1\rangle\rangle.$ Now, if $a_1+\langle z_1\rangle=0,$ then  $x_1+\langle z_1\rangle=\alpha'b+\langle z_1\rangle\in I_3/\langle z_1\rangle$ and so $[q,x_1]+\langle z_1\rangle=0.$ Since $[q,x_1]=z_2,$ we have $[q,x_1]\in \langle z_1\rangle\cap \langle z_2\rangle=0$ and hence $[q,x_1]=0.$ It is a contradiction. Thus $a_1+\langle z_1\rangle \neq 0$ and $x_1-a_1-\alpha'b\in \langle z_1\rangle.$ Taking into account the basis of $L,$ we can write $a_1=\beta_1x_1+\beta_2x_2+\beta_3y_1+\beta_4y_2+\beta_5z_1+\beta_6z_2+\beta_7q$ and $\alpha'b=\beta'_1x_1+\beta'_2x_2+\beta'_3y_1+\beta'_4y_2+
 \beta'_5z_1+\beta'_6z_2+\beta'_7q.$ Therefore $x_1=a_1+\alpha'b+\gamma z_1=(\beta_1+\beta'_1)x_1+(\beta_2+\beta'_2)x_2+(\beta_3+\beta'_3)y_1+
 (\beta_4+\beta'_4)y_2+(\beta_5+\beta'_5)z_1+(\beta_6+\beta'_6)z_2+(\beta_7+\beta'_7)q.$ Now, since the set $\{x_1,y_1,x_2,y_2,z_1,z_2,q\}$ is linearly independent and the characteristic of the field is $2,$ we have $\beta_1=1-\beta'_1,$ $\beta_2=\beta'_2,$ $\beta_3=\beta'_3,$ $\beta_4=\beta'_4,$ $\beta_5=\beta'_5+\gamma,$ $\beta_6=\beta'_6$ and $\beta_7=\beta'_7.$ Now we have $x_1+\langle z_1\rangle=a_1+\langle z_1\rangle+\alpha'b+\langle z_1\rangle=(\beta_1+\beta'_1)x_1+(\beta_2+\beta'_2)x_2+(\beta_3+\beta'_3)y_1+(\beta_4+\beta'_4)y_2+(\beta_6+\beta'_6)z_2+(\beta_7+\beta'_7)q+\langle z_1\rangle=(\beta_1+\beta'_1)x_1+\langle z_1\rangle.$ Therefore $\beta_1\neq 0$ or $\beta'_1\neq 0$. But $x_1+\langle z_1\rangle\notin I_4/\langle z_1\rangle$ and hence $\beta_1\neq 0$ and $\beta'_1=0$ which implies $x_1+\langle z_1\rangle\in I_3/\langle z_1\rangle.$ Thus $ x_1\in I_3. $ \\
  Now, we want to show $[x_1,x_2]=[x_2,q]=[y_2,q]=[y_1,q]=[x_1,y_2]=0.$  By contrary, let $ \alpha_9\neq 0. $ Using the same procedure as used in the previous step, we can show $x_2\in I_3.$ Because $I_3/\langle z_1\rangle\cong H(2),$ we have $[x_2,q]+\langle z_1\rangle=0$ and so $[x_2,q]\in \langle z_1\rangle\cap \langle z_2\rangle=0$ which is a contradiction. Similarly, one can show that each of the mentioned brackets vanishes.\\
  Here, we want to prove that exactly one of the brackets $[x_2,y_1]$ or $[y_1,y_2]$ vanish. Note that the case $[x_2,y_1]=[y_1,y_2]=0$ leads to $L/\langle z_1\rangle\cong H(1)\oplus A(3)$ which is a contradiction since $L/\langle z_1\rangle\cong H(2)\oplus A(1).$ Now, without loss of generality, assume that $[y_1,y_2]\neq 0.$ Using the same process, we may prove $[x_2,y_1]=0$ and $y_1,y_2\in I_3$ and $ x_2\in I_4. $ Thus $ L=\langle x_1,y_1,x_2,y_2,z_1,z_2,q|[y_1,y_2]=z_2=[x_1,q],[x_1,y_1]=z_1=[x_2,y_2] \rangle .$ Similarly if either $ L/\langle z_2\rangle\cong H(2)\oplus A(1)$ and $ L/\langle z_1\rangle \cong H(1)\oplus A(3) $ or $ L/\langle z_2\rangle\cong H(1)\oplus A(3)$ and $ L/\langle z_1\rangle \cong H(1)\oplus A(3),$ then $ L\cong \langle  x_1,\ldots, x_7\big{|}[x_1, x_2] = x_6, [x_1, x_4] = x_7=[x_3, x_5] \rangle. $ The result  follows.
\end{proof}
\begin{cor}\label{jhk6}
Let $ L $ be a $7$-dimensional generalized Heisenberg Lie algebra over any filed $F.$  Then  $L\cong  \langle x_1,\ldots, x_7\big{|}[x_1, x_2] = x_6=[x_3, x_4],[x_1, x_5] = x_7= [x_2, x_3]\rangle \cong L_1$ or $L\cong \langle  x_1,\ldots, x_7\big{|}[x_1, x_2] = x_6, [x_1, x_4] = x_7=[x_3, x_5] \rangle\cong L_2. $
\end{cor}
\begin{proof}
The result follows from Theorems \ref{11233} $ (ii) $ and \ref{klj1}.
\end{proof}
The following result gives the Schur multipliers of Lie algebras $L_{6,22}(\epsilon)  $ and $L_{6,7}^{(2)}(\eta).$ It helps  to determine the  capability of these Lie algebras in the next proposition.

\begin{prop}\label{mul}
The Schur multiplier of Lie algebras $L_{6,22}(\epsilon)  $ and $L_{6,7}^{(2)}(\eta)  $ are abelian Lie algebras of dimension $ 8.$
\end{prop}
\begin{proof}
Using the method of Hardy and Stitzinger in \cite{ha}, we can obtain that in both cases, the dimension of the Schur multipliers are  $ 8.$
\end{proof}

In the characterizing of capable generalized  Heisenberg Lie algebra of rank two  of dimension $ 6 $ in \cite[Theorem 2.12]{ni4}, the Lie algebra $L_{6,22}(\epsilon)  $ is missing. Here,
we  improve this result as bellow.
\begin{prop}\label{11}
$L_{6,22}(\epsilon)  $ and $L_{6,7}^{(2)}(\eta)  $  are capable.
\end{prop}
\begin{proof}
 Let $L\cong L_{6,22}(\epsilon) $ is capable. Then by using Theorem \ref{11233}, we have $L_{6,22}(\epsilon)^2 = Z( L_{6,22}(\epsilon))= \langle x_5\rangle \oplus \langle x_6\rangle.$ By Proposition \ref{dd} $(ii) $ and \cite[Corollary 4.6]{alam}, it is enough to show that $ \dim   \mathcal{M}(L_{6,22}(\epsilon)/\langle x_i\rangle) -1< \dim \mathcal{M}(L_{6,22}(\epsilon)) $ for $ i=5,6. $ Clearly, $ L_{6,22}(\epsilon)/\langle x_i\rangle \cong H(2) $ or $ L_{6,22}(\epsilon)/\langle x_i\rangle \cong H(1)\oplus A(2)$ for $ i=5,6, $ by using \cite[Theorem 3.6]{ni3}. Thus for $ i=5,6,$ we have $ \dim   \mathcal{M}(L_{6,22}(\epsilon)/\langle x_i\rangle)=5~ \text{or}~  8,$ by \cite[Lemma 2.6 and Theorem 2.7]{ni3}.
Since  $\dim \mathcal{M}(L_{6,22}(\epsilon))=8,$ by Proposition \ref{mul}, we conclude that  \begin{align*}
   \dim \mathcal{M}(L_{6,22}(\epsilon)/\langle x_i\rangle)-1< \dim \mathcal{M}(L_{6,22}(\epsilon))~\text{for}~  i=5,6.\end{align*}
Therefore $L_{6,22}(\epsilon)  $  is capable. By a similar way, we can see that $L_{6,7}^{(2)}(\eta)  $ is also  capable. The proof is completed.
\end{proof}
The next result is useful.
\begin{lem}\cite[Lemma 2.11]{ni4}\label{111}
 $  L_{5,8} $ and   $ L_1$ are capable while $ L_2 $ is not.
\end{lem}
We are ready to summarize  our results to show that which ones of generalized Heisenberg Lie algebras of rank $2$  is capable.

\begin{thm}\label{112}
Let $H$ be an $ n$-dimensional generalized Heisenberg Lie algebra of rank $2.$ Then
 $ H $  is capable if and only if
  $ H $ is isomorphic to one of Lie algebras $ L_{5,8}, L_{6,22}(\epsilon), L_{6,7}^{(2)}(\eta)$ or  $ L_1. $
 \end{thm}
\begin{proof}
Let $ H $ be capable. Then \cite[Proposition 2.6]{ni4} implies $5\leq  \dim H\leq 7.$ Using Corollary \ref{jhk6} and  Theorem \ref{11233},  $ H $ is isomorphic to one of Lie algebras $ L_{5,8}, L_{6,22}(\epsilon),L_{6,7}^{(2)}(\eta),$  $ L_1 $ or $ L_2. $  By Proposition \ref{11} and  Lemma \ref{111}, all of them are capable while $ L_2 $ is non-capable. The converse is held by Proposition \ref{11} and  Lemma \ref{111}.
\end{proof}

 \section{Stem nilpotent Lie algebras of class  $ 3 $  with the derived subalgebra of dimension $ 2 $}
We know that every nilpotent Lie algebra with the derived subalgebra of dimension 2 is of class two and three.
  In this section, we are going to obtain the structure of stem Lie algebras of class $ 3 $ with the derived subalgebra of dimension $ 2. $ Then we determine which of them is capable. Moreover, we show that all such Lie algebras of dimension greater than $ 6 $ are unicentral. \\
 Recall that an $n$-dimensional nilpotent Lie algebra  $L$  is said to be  nilpotent of maximal class if the class of $L$ is $n-1.$  If $L$ is of maximal class, then  $\dim (L/L^2) = 2,$ $Z_i(L) = L^{n-i}$ and $\dim (L^j/L^{j+1}) = 1$ for all $0\leq  i \leq n-1$ and $2\leq j \leq n-1$      (see \cite{bos} for more information).

 From \cite{Gr}, the only Lie algebra of  maximal class of dimension $4$ is isomorphic to
$L_{4,3}=\langle x_1,\ldots,x_4|[x_1, x_2] = x_3, [x_1, x_3] = x_4\rangle.$

We say a Lie algebra $L$ is a semidirect sum of an ideal $I$ and a subalgebra $K$ if $L=I+K,$
$ I\cap K=0. $ The semidirect sum of an ideal $I$ and a subalgebra $K$ is denoted by $K\ltimes I.$\newline

Let $ cl(L)$ denotes the nilpotency class of a Lie algebra $ L. $
The following two lemmas characterize the structure of all stem Lie algebras $ L $ of dimensions $ 5$ and $6, $ when $ cl(L)=3 $ and $ \dim L^2=2. $

 \begin{lem}\label{rr11}
  Let $L  $  be a nilpotent stem Lie algebra of dimension $ 5 $ such that $ \dim L^2=2 $ and $ cl(L)= 3.$ Then
\[ L\cong L_{5,5}=\langle x_1,\ldots,x_5| [x_1, x_2] = x_3, [x_1, x_3] = x_5, [x_2, x_4] = x_5\rangle.\]
Moreover,
$L_{5,5}=I\rtimes \langle x_4\rangle  $, where \[ I=\langle x_1,x_2,x_3,x_5| [x_1, x_2] = x_3, [x_1, x_3] = x_5\rangle\cong L_{4,3},~\text{and}~[I, \langle x_4\rangle]=\langle x_5\rangle.\]
  \end{lem}
\begin{proof}
 By the classification  of $5$-dimensional nilpotent Lie algebras  in \cite{Gr}, we get $L\cong L_{5,5}.$
It is easy to check that $L_{5,5}=I\rtimes \langle x_4\rangle  $ such that  $ I=\langle x_1,x_2,x_3,x_5| [x_1, x_2] = x_3, [x_1, x_3] = x_5\rangle\cong L_{4,3}$  and $[I, \langle x_4\rangle]=\langle x_5\rangle.$
\end{proof}
\begin{lem}\label{rr112} Let $L  $  be a nilpotent stem Lie algebra of dimension $ 6 $ such that $ \dim L^2=2 $ and $ cl(L)= 3.$ Then
$L\cong L_{6,10}=\langle x_1,\ldots,x_6| [x_1, x_2] = x_3, [x_1, x_3] = x_6, [x_4, x_5] = x_6\rangle.$
Moreover,
$L_{6,10}=I\dotplus \langle  x_4,x_5,x_6|[x_4, x_5] = x_6 \rangle=I\dotplus K  $ such that  $I=\langle  x_1, x_2, x_3,x_6|[x_1, x_2] = x_3, [x_1, x_3] = x_6 \rangle\cong L_{4,3}$
and $K=\langle  x_4,x_5,x_6|[x_4, x_5] = x_6 \rangle\cong H(1).$
  \end{lem}
  \begin{proof}
  By the classification  of $6$-dimensional nilpotent Lie algebras  in \cite{Gr}, we get $L\cong L_{6,10}.$
Clearly $ Z(L)= \langle x_6\rangle$ and $L_{6,10}=I+K,$ where $I=\langle  x_1, x_2, x_3,x_6|[x_1, x_2] = x_3, [x_1, x_3] = x_6 \rangle\cong L_{4,3}$ and $K=\langle  x_4,x_5,x_6|[x_4, x_5] = x_6 \rangle \cong H(1).$ Since $ I\cap K= \langle x_6\rangle=Z(I)=Z(L)$ and $[I, K]=0,  $
 we can see $L_{6,10}=I\dotplus K.$
\end{proof}
The following  proposition is an useful instrument in the next.
\begin{prop}\label{48}
Let $L$ be an $n$-dimensional nilpotent stem Lie algebra  of class $3$  $(n\geq 5)$ and $ \dim L^2=2 $ such that $L=I+K,$ where $I$ and $ K $ are two subalgebras of $L,$  $I\cong L_{4,3}$ is the maximal class Lie algebra of dimension $ 4 $ and  $[I,K]\subseteq Z(I)=Z(L).$ Then
\begin{itemize}
\item[$(i)$] If $K$ is a non-trivial abelian Lie algebra such that $K\cap I=0,$ then
$[I,K]=Z(L)  $ and $K\cong A(1).$ Moreover, $L=I\rtimes K\cong  L_{5,5}.$

\item[$(ii)$] Assume  $\dim K^2=1$ and $I\cap K=K^2=Z(L).$
\begin{itemize}
\item[$(a)  $]If $ K^2=Z(K),$ then  $L=I\dotplus K,$ where $n=2m+4.$  Moreover, for $ m=1, $ we have $L=I\dotplus K\cong L_{6,10},$ where $ K\cong H(1). $  For $ m\geq 2, $ we have $L=I\dotplus K\cong I_1\dotplus I_2,$ where $ I_1\cong L_{6,10} $ and $ I_2\cong H(m-1).$
\item[$(b)  $]If $ K^2\neq Z(K),$ then $L=(I\rtimes A)\dotplus K,$ where  $ K\cong H(m) $ and $A\cong A(1),$  $ [I,A]=Z(L)=Z(I)=K^2$ and $ n=2m+5.$ Moreover, $ I\rtimes A\cong L_{5,5}. $
\end{itemize}

\end{itemize}
\end{prop}
\begin{proof}
\begin{itemize}
\item[ $(i)$]  We have  $[I,K]\subseteq Z(I)=Z(L),$ so $ I$ is an ideal of $L.$ Since $ K\cap I=0 $ and $I\cong  L_{4,3},$  $ \dim K=\dim L-\dim I=n-4 $ and so $L=I\rtimes K,$  where $K\cong A(n-4).$ We claim that $ [I,K ]=Z(I).$  By contrary, assume that  $ [I,K ]=0.  $ Since $ K $ is abelian, we have $ K\subseteq Z(I)=Z(L)\subseteq I.$ Now  $ I\cap K\neq 0,$ so we have  a contradiction. Thus $ [I,K ]=Z(I)=Z(L).$ We know that $ I $ is a Lie algebra of maximal class of dimension $ 4 $ so $ Z(I)=Z(L)=I^3. $  We also have $ \dim L^2=\dim I^2=2. $ Therefore  $ L^2=I^2. $ We claim that $K\cong A(1).  $ \newline
 First assume that $ n=5.$  Lemma \ref{rr11} implies $ L\cong I\rtimes K\cong L_{5,5}$  and since $I\cong L_{4,3},  $ we have $I=\langle x_1,\ldots,x_4|[x_1, x_2] = x_3, [x_1, x_3] = x_4\rangle$ and $Z(I)=\langle x_4\rangle=I^3.$
 Now let $ n\geq 6$  and $ K=\bigoplus_{i=1}^{n-4}\langle z_i\rangle. $ In this case we show that $ K\cap I\neq 0, $ which is a contradiction. So this case does not occur. By using Jacobian identity, for all  $ 1\leq i\leq n-4, $ we have
 \begin{align*}
 [z_i,x_3]=[z_i,[x_1, x_2]]=[z_i,x_1, x_2]+[x_2,z_i, x_1]=0
 \end{align*}
 since $ [z_i,x_1]$  and $[x_2,z_i] $ are central. Thus $[z_i,x_3]= 0$ for all $ 1\leq i\leq n-4.$

 Now, let $[z_i,x_1]= \alpha_i x_4$ with $ \alpha_i\neq  0. $ Putting $ z_i'=z_i+\alpha_i x_3,$ we have $[z_i',x_1]=[z_i+\alpha_i x_3,x_1]=[z_i,x_1]+\alpha_i [x_3,x_1]= \alpha_i x_4-\alpha_i x_4=0. $ So we also obtain $ [z_i',x_3]=0.$ Thus $[z_i',x_3]=[z_i',x_1]= 0$ for all $i$, $ 1\leq i\leq n-4.$
  Now let $ [z_i',x_2]= \alpha x_4$ and  $ [z_j',x_2]=\beta x_4,$ in which $ \alpha\neq 0 $ and $ \beta\neq 0 $ for $ i\neq j$ and fixed $ i,j$.
 Put $ d_i=\beta z_i'-\alpha z_j'.$ We have
 $[d_i,x_2]=[\beta z_i'-\alpha z_j',x_2]=\beta \alpha x_4-\beta \alpha x_4=0$
and so  $[d_i,x_2]=0.$ On the other hand, $ [d_i,x_1]=[d_i,x_2]=[d_i,x_3]=0. $ Therefore $ [d_i,I]=0 $ and hence $ d_i\in Z(L)=Z(I)= \langle x_4\rangle.$ Since
 \begin{align*}
 &d_i=\beta z_i'-\alpha z_j'=\beta(z_i-\alpha_i x_3)-\alpha ( z_j-\alpha_j x_3)\\&=\beta z_i- \alpha z_j + (\alpha\alpha_j-\beta\alpha_i)x_3\in Z(I)
 \end{align*}
so  $0\neq  \beta z_i- \alpha z_j\in K\cap I=0,$ which is a contradiction.  Thus $ n=5, $ $ K\cong A(1) $ and  $L=I\rtimes \langle z_1\rangle$ and $ [x_2,z_1]=x_4, $ as required. Considering the classification of nilpotent Lie algebras of dimension $ 5$ with $ \dim L^2=2 $ given in \cite{Gr} and Lemma \ref{rr11}, we should have $L\cong L_{5,5}.$
\item[$ (ii) $]
Since
 $I\cap K=K^2=Z(L)=Z(I)\cong A(1),$  $\dim (K)=\dim (L)-\dim (I)+\dim (I\cap K)=n-4+1=n-3. $ We know that $ \dim  K^2=1, $ so \cite[Theorem 3.6]{ni3} implies  $ K\cong K_1\oplus A, $ in which $K_1\cong H(m)$ and $A\cong A(n-2m-4).$ If  $ A=0,$ then $K \cong H(m).$
  Since
 $I\cap K=K^2=Z(L)=Z(I)=Z(K)\cong A(1),$  $\dim (K)=\dim (L)-\dim (I)+\dim (I\cap K)=n-4+1=n-3.  $  Now since   $2m+1=\dim (K)=n-3,$ we have $ n=2m+4. $ We are going to show that  $ [I,K]=0. $ In fact, we show that there exists $ I_1\cong L_{3,4} $ and $ K_2\cong H(m)$ with $ [I_1,K_2]=0 $ and $ L=I_1\dotplus K_2. $
 First let $ m=1. $ We have $ \dim L=6 $ and $ K=\langle x,y,x_4|[x,y]=x_4\rangle, $ since $ K\cong H(1).$ By looking the classification of nilpotent Lie algebras of dimension $ 6$ with $ \dim L^2=2 $ given in \cite{cic} and Lemma \ref{rr112},  we should have $L\cong L_{6,10}.$
  Now, let $ m\geq 2$ and $ H(m)=\langle a_1,b_1,\ldots, a_m,b_m,z\big{|}[a_l,b_l]=z,1\leq
l\leq m\rangle.$ Lemma \ref{fr} implies that $H(m)=T_1\dotplus \ldots \dotplus T_m,$ in which  $T_i\cong H(1) $
 for all  $ 1\leq i \leq m.$ With the same procedure as case in $ (i) $ and changing the variables we can see that $ [T_i,I]= 0$ for all $i$, $ 1\leq i\leq m.$
So $[I,K]=0$ and hence  $ L=I \dotplus K.$ Since $ m\geq 2, $ we have $ L=(I\dotplus T_1)\dotplus (T_2 \dotplus \ldots \dotplus T_m)$ such that $ I\dotplus T_1\cong L_{6,10}$ and $T_2 \dotplus \ldots \dotplus T_m\cong H(m-1),  $ as required.
The case $ (a) $ is completed.\\
 Now, let $ A\neq 0$ and so  $n\neq 2m-4.$ Thus $L=I+( K_1\oplus A)$ such that $[I,K]\subseteq Z(L)=Z(I).$ We are going to show that $ A\cong A(1), $ $ [I,K_1]=0 $ and $ [I,A]=Z(I)=Z(L). $ Similar to the part $ (ii), $ we can see that  $[I,K_1]=0.$ We claim that  $ [I,A]\neq 0.  $ By contrary, let  $[A,K_1]=[I,A]=0  $ and so $ A\subseteq Z(L)=Z(I). $ Since  $ A\cap I=0,$ we have $ A=0, $ which is a contradiction. So we have $ [I,A]=Z(L).$ We claim that $ \dim A=1. $ Let  $ \dim A\geq 2. $ Similar to the proof of the  part $ (i),$ we have    $ [a_1,x_1]=[a_2,x_1]=[a_1,x_3]=[a_2,x_3]=0$  where $ a_1,a_2\in A $ and $a_1\neq a_2.  $ Now let $[a_1,x_2]= \alpha x_4$ and  $ [a_2,x_2]=\beta x_4$ such that $ \alpha\neq 0 $ and $ \beta\neq 0. $
 Putting $ a_1'=\beta a_1-\alpha a_2.$ We have
 \begin{align*}
 [a_1',x_2]=[\beta a_1-\alpha a_2,x_2]=\beta \alpha x_4-\beta \alpha x_4=0
 \end{align*}
and so  $[a_1',x_2]=0.$ Hence $ [a_1',x_1]=[a_1',x_2]=[a_1',x_3]=0. $ Therefore $ [a_1',I]=0 $ and hence $ a_1'\in Z(L)=Z(I)= \langle x_4\rangle=K_1^2.$ So $ a_1'\in K_1 $ and since $ K_1\cap A=0,$ we have  a contradiction. Hence $ A\cong A(1) $ and so $n=2m+5.$
   Thus $ L=(I\rtimes A)\dotplus K_1 $ such that  $[I,A]=Z(L)=Z(I).$ By part $(i),$ we have $ I\rtimes A\cong L_{5,5}. $ The case $ (b) $ is completed. The result follows.
\end{itemize}
\end{proof}
We need the following lemma for the next investigation.
\begin{lem}\label{z}\cite[Lemma 1]{K} Let $L$ be a nilpotent Lie algebra and $H$ be a subalgebra of $L$ such that $L^2 =
H^2 + L^3.$ Then
$L^i = H^i$ for all $i \geq 2.$
Moreover, $H$ is an ideal of $L.$
\end{lem}
In the following, we determine the central factor of all stem Lie algebras $ T $ such that $cl(T)=3$ and $\dim T^2=2.$

  \begin{lem}\label{ggg}
   Let $ T $ be an $ n $-dimensional stem Lie algebra   such that $cl(T)=3$ and $\dim T^2=2.$ Then $ Z(T)=T^3\cong A(1)$ and $T/Z(T)\cong H(1)\oplus A(n-4).$
  \end{lem}
  \begin{proof}
  Since $T$ is stem, we have  $  A(1)\cong T^3  \subseteq Z(T)\subsetneqq T^2.$ Thus
$ Z(T)=T^3\cong A(1).$ This follows  $T^2/Z(T)\cong  A(1).$
Since $ T/Z(T) $ is capable, \cite[Theorem 3.6]{ni3} implies that $T/Z(T)\cong H(1)\oplus A(n-4).$ It completes the proof.
  \end{proof}
 In the following theorem, we determine the structure of all stem Lie algebras of class $ 3 $ with the derived subalgebra of dimension $ 2.$
  \begin{thm}\label{lkl}
    Let $ T $ be an $ n $-dimensional stem Lie algebra   such that $cl(T)=3$ and $\dim T^2=2.$ Then one of the following holds
\begin{itemize}
\item[$(a)$] $ T\cong L_{4,3}.$
\item[$ (b) $]
$ T\cong  I\rtimes K \cong L_{5,5}$ where $K\cong A(1),$ $I\cong  L_{4,3} $ and $Z_2( T) =Z_2( I)\rtimes K. $
\item[$ (c) $]
$ T\cong  I\dotplus I_1,$ where $ I_1\cong H(m), $
$Z_2( T) =Z_2( I)+I_1,$ $I\cong  L_{4,3} $ and $n=2m+4.$ Moreover if $ m\geq 2, $ then $L\cong L_{6,10} \dotplus I_2,$ where $I_2\cong H(m-1),$ if $ m=1, $ then $L\cong L_{6,10}.$
\item[$ (d )$]
$ T\cong ( I\rtimes  K)\dotplus I_1\cong L_{5,5}\dotplus I_1  ,$ where $ I_1\cong H(m),~K\cong A(1),$ $I\cong  L_{4,3}, $
$Z_2( T) =(Z_2( I)\rtimes  K)\dotplus I_1$  and $n=2m+5.$
\end{itemize}
Moreover, in the cases  $ (b), $ $(c) $ and  $(d),$  $ Z(T)=Z(I)=I_1^2=[I,K]. $
  \end{thm}
\begin{proof}
Since $ cl(T)=3, $ we have $\dim T\geq 4.  $ If $\dim T=4,$ then $ T $ must be a Lie algebra of maximal class and hence $T\cong L_{4,3}.$  Assume that $\dim T\geq 5.$ We have $T/Z(T)\cong H(1)\oplus A(n-4)$ and $Z(T)=T^3\cong A(1),$ by Lemma \ref{ggg}.
There exist  ideals $ I_1/Z(T)$ and $ I_2/Z(T)$   of $ T/Z(T)$ such that
\[  I_1/Z(T)\cong H(1) ~ \text{and}~
I_2/Z(T)\cong A(n-4).\]
Since $ T^2/Z(T)=\big{(}I_1^2+Z(T)\big{)}/Z(T),  $ we have $ T^2=I_1^2+Z(T)  $ and $ Z(T)=T^3.$
Using Lemma \ref{z}, we have $ T^2=I_1^2 $ and so $cl(T)=cl(I_1)=3.$ Hence $ I_1$ is a Lie algebra of maximal class and since  $\dim I_1=4,$  we have  $I_1\cong L_{4,3}.$ Now, $ Z(T)=Z(I_1)$ because $Z(T)\cap   I_1\subseteq Z( I_1)  $ and $ \dim Z(T)=1. $ Since $ Z(T)\subseteq   I_1\cap I_2 \subseteq Z(T),$ we have $ I_1\cap I_2= Z(T)=Z(I_1). $ Now we are going to determine the structure of $ I_2. $ We have $I_2/Z(T)\cong A(n-4)$ so  $I_2^2\subseteq Z(T)\cong A(1),$ and hence  $cl(I_2) \leq 2$ and $ [I_1,I_2]\subseteq I_1\cap I_2 = Z(T)=Z(I_1)\cong A(1).$ We have $ \dim T/Z(T)\geq 4 $ and so $ \dim I_2\geq 2. $
  Let $cl(I_2)=1.$ Therefore $[I_1,I_2]=I_1\cap I_2=Z(T),$ otherwise $[I_1,I_2]=0$  and since $  I_2$ is abelian,  $ I_2\subseteq Z(T)\cong A(1).$ It  is a contradiction, since $ \dim I_2\geq 2. $ Hence  $ I_2=Z(T) \oplus A$, where $A\cong A(n-4)$ and $[I_1,I_2]=Z(T).$ Now $ Z(T)\subseteq I_1, A\cap I_1=0 $ and $ I_1\cap I_2 = Z(T) $ so $ T= I_1+I_2=I_1+Z(T)+A=I_1\rtimes A.$ Using the proof of Proposition \ref{48} $ (i),$  we have $ T\cong I_1\rtimes K\cong L_{5,5} $ in which $K\cong A(1)$ and $ [K,I_1]=Z(T). $ This is  the case $ (b). $ \\ Now, let $ cl(I_2)=2.$ Since $I_2^2= I_1\cap I_2=Z(T)=Z(I_1)\cong A(1),$  by \cite[Theorem 3.6]{ni3}, we have $I_2\cong H(m)\oplus A(n-2m-4).$ First assume that $ A(n-2m-4)=0.$ Then $n=2m-4$ and $ I_2\cong H(m).$ Using Proposition \ref{48} $ (ii) (a),$ we can similarly prove $[I_1,I_2]=0$  and $ T=I_1\dotplus I_2 $ where $ I_2\cong H(m).  $ This is the case $ (c). $
   Now, let  $  A(n-2m-4)\neq 0. $ Then $n\neq 2m-4$ and hence $T=I_1+(K\oplus A)$ where $K\cong H(m)$ and  $A\cong A(n-2m-4)$ and $[I_1,K\oplus A]\subseteq Z(T)=Z(I_1).$ Similar to the case $ (c), $ we have  $[I_1,K]=0.$ Now we claim that $ [I_1,A]=Z(T)\cong A(1). $ Let $ [I_1,A]=0. $ Since $ [K,A]=0, $ we have  $ A\subseteq Z(T)=Z(I_1)=Z(K)=K^2\cong A(1).$ It is a contradiction, since $ A\cap K=0. $ Therefore $[I_1,A]=Z(T)  $ and hence $ T\cong (I_1\rtimes A)\dotplus K $ where $ A\cong A(n-2m-4)$ and  $ [I_1,A]=Z(T)=Z(I_1).$ Similar to the case $ (b), $ one can obtain that $ A\cong A(1), $ so $ n-2m-4=1, $  $n=2m+5  $ and  $ [I_1,A]=Z(T). $ So $ T=(I_1\rtimes A)\dotplus K $ in which $ A\cong A(1) $ and $ [I_1,A]=Z(T). $ This is the case $ (d). $\\ Now, we have
 \[Z_2(T)/Z(T)=Z(T/Z(T))=Z(I_1/Z(T))\oplus I_2/Z(T~)\text{ and}~ Z(T)=Z(I_1)\] also $ Z(I_1/Z(T))=I_1^2/Z(T),$ so $Z_2(T)/Z(T)=I_1^2/Z(T)\oplus I_2/Z(T).  $   Since $ I_1 $ is maximal class of dimension $ 4, $ we have
   $Z_2( T) =Z_2( I_1)+I_2=I_1^2+I_2.$ The result follows.
\end{proof}
In the following theorem, we classify all non-capable stem Lie algebras of class $ 3 $ with the derived subalgebra of dimension $ 2.$
\begin{thm}\label{171}
Let $ T $ be an $ n $-dimensional stem Lie algebra   such that $cl(T)=3,$  $\dim T^2=2$ and $n\geq 6.$ Then $T\cong( I\rtimes  K)\dotplus H$  or $ T\cong I\dotplus H$ such that $H\cong H(m), K\cong A(1),$  $I\cong L_{4,3}$ and  $[K ,I]= Z(T)=Z(I)=H^2.$ Moreover, $T$ is non-capable.
\end{thm}
\begin{proof}
By Theorem \ref{lkl} $(c) $ and  $(d) $, we obtain $( I\rtimes  A)\dotplus H$  or $ I\dotplus H$ such that $H\cong H(m), A\cong A(1),$ $I\cong L_{4,3}$ and  $[A ,I]= Z(T)=Z(I)=H^2.$
 By using  Proposition \ref{Hi},
$T$ is non-capable.
 The result follows.
\end{proof}
The capable stem Lie algebras of class $ 3 $ with the derived
subalgebra of dimension $ 2 $ are characterized as following.
\begin{lem}\label{cc1}
$L_{4,3}  $ and $ L_{5,5} $ are  capable.
\end{lem}
\begin{proof}
 From  \cite{Gr}, let
$L_{5,7}=\langle x_1,\ldots,x_5| [x_1, x_2] = x_3, [x_1, x_3] = x_4, [x_1, x_4] = x_5\rangle$ and
$L_{6,13}=\langle x_1,\ldots,x_6| [x_1, x_2] = x_3, [x_1, x_3] = x_5, [x_2, x_4] = x_5, [x_1, x_5] = x_6,[x_3, x_4] = x_6\rangle$.
We have $Z(L_{5,7})=\langle x_5\rangle $ and $Z(L_{6,13})=\langle x_6\rangle,  $ so $L_{5,7} /\langle x_5\rangle\cong L_{4,3}$  and $L_{6,13}/\langle x_6\rangle\cong L_{5,5}.$  Thus $L_{4,3}  $ and $ L_{5,5} $ are  capable.
\end{proof}
We are in a position to characterize the capability of an $ n$-dimensional stem Lie algebra $ T $  such that $cl(T)=3$  and $\dim T^2=2.$
\begin{thm}\label{p11}
  Let $ T $ be an $ n$-dimensional stem Lie algebra such that $cl(T)=3$  and $\dim T^2=2.$  Then $ T$ is capable if and only if  $T\cong L_{4,3}$ or   $T\cong L_{5,5}.$
  \end{thm}
\begin{proof}
  Let $ T $ be  capable. By Theorems  \ref{lkl}, \ref{171} and Lemma \ref{cc1}, $ T $ is isomorphic to $ L_{4,3}$ or  $ L_{5,5}.$  The converse holds by Lemma \ref{cc1}.
\end{proof}
The next theorem gives a necessary and sufficient condition for detecting  the capability of stem Lie algebras of class $ 3 $ with the derived subalgebra of dimension $ 2. $
\begin{thm}\label{fg}
Let $ T $ be an $ n$-dimensional stem Lie algebra such that $cl(T)=3$  and $\dim T^2=2.$ Then $ T$ is capable if and only if
$3\leq \dim (T/Z(T)) \leq 4.$
\end{thm}
\begin{proof}
The result follows from Lemma \ref{ggg} and Theorem \ref{p11}.
\end{proof}
Recall that a Lie algebra $L$ is called  unicentral if $Z^{*}(L)=Z(L).$
\begin{cor}\label{gk}
Let $ T $ be an $ n$-dimensional stem Lie algebra such that $cl(T)=3$  and $\dim T^2=2.$ Then $ T $ is non-capable if and only if $ n\geq 6.$
Moreover, $T $ is unicentral.
\end{cor}
\begin{proof}
The result follows from Theorems \ref{171} and \ref{p11}.
\end{proof}

 \section{Nilpotent Lie Algebras with the derived subalgebra of dimension two}
 In this section, we are going to determine all capable nilpotent Lie algebras   with the derived subalgebra of dimension $ 2. $
At first we show that every finite dimensional nilpotent Lie algebra of  class
 $3$ with derived subalgebra of dimension $2$ can be considered as a direct sum of a non-abelian stem Lie algebra of class $3$ and an abelian Lie algebra.\\
 The following result shows that the capability of the direct product of a  non-abelian Lie algebra and an  abelian Lie algebra depends only on the capability of its non-abelian factor.
\begin{thm}\label{8}
Let $ L $ be a finite dimensional nilpotent Lie algebra of  class $ 3$ and $ \dim L^2=2 $. Then $ L=T\oplus A $ such that  $ Z(T) =L^2\cap Z(L)=L^3=T^3$  and  $Z^{*}(L)=Z^{*}(T), $ where $ A $ is an abelian Lie algebra.
\end{thm}
\begin{proof}
By using \cite[Proposition 3.1]{pair}, $ L=T\oplus A $ such that  $ Z(T) =L^2\cap Z(L)$  and  $Z^{*}(L)=Z^{*}(T), $ where $ A $ is an abelian Lie algebra. Since $ T $ is stem, Lemma \ref{ggg} implies $ Z(T) =T^3,$ as required.
\end{proof}
In the following theorem, all capable nilpotent Lie algebras of class $ 2 $ with the derived subalgebra of
dimension $ 2$   are classified.
\begin{thm}\label{hh}
Let $ L $ be an $n$-dimensional nilpotent Lie algebra of nilpotency class $ 2 $ and  $ \dim L^2=2. $   Then $ L$ is capable if and only if
$ L\cong L_{5,8} \oplus A(n-5)$, $ L\cong L_{6,22}(\epsilon) \oplus A(n-6)$, $ L\cong L_{6,7}^{(2)}(\eta) \oplus A(n-6)$ or $ L\cong L_1 \oplus A(n-7).$
\end{thm}
\begin{proof}
This is immediately obtained from \cite[Propositions 2.4 and 2.6]{ni4} and  Theorem \ref{112}.
\end{proof}

We are ready to determine all  capable Lie algebras  of class $3  $ when its derived subalgebra is of dimension $ 2. $
\begin{thm}\label{2611}
Let $ L $ be an $n$-dimensional Lie algebra    such that $cl(L)=3$  and $\dim L^2=2.$ Then $L $ is capable if and only if
 $L\cong L_{4,3}\oplus  A(n-4) $ or  $L\cong L_{5,5}\oplus  A(n-5).$
\end{thm}
\begin{proof}
 Theorem \ref{8} implies  $L\cong T\oplus A,$ where $ A $ is an abelian Lie algebra and  $ Z(T)=T^2\cap Z(L)=L^3=T^3\cong A(1)$ and $Z^{*}(L)=Z^{*}(T). $ Now the result follows from Theorem \ref{p11}.
\end{proof}
The following result is obtained from Theorems \ref{fg} and \ref{2611}.
\begin{cor}\label{fg1}
Let $ L$ be a finite dimensional Lie algebra  of class $3  $ and $\dim L^2=2.$ Then $ L$ is capable if and only if
$3\leq \dim (L/Z(L)) \leq 4.$
\end{cor}
We summarize all results  to classify  all capable  nilpotent Lie algebras with the derived subalgebra of dimension at most two.
\begin{thm}
Let $ L $ be an $ n$-dimensional  nilpotent Lie algebra  with $\dim L^2\leq 2.$  Then $ L $ is capable is if and only if $ L $ is isomorphic to one the following Lie algebras.
\begin{itemize}
\item[(i)] If $\dim L^2=0,$ then $ L\cong A(n) $ and $ n>1.$
 \item[(ii)] If $\dim L^2=1,$ then $ L\cong H(1)\oplus A(n-3).$
\item[(iii)]If $\dim L^2=2$ and $cl(L)=2,$ then $ L\cong L_{5,8} \oplus A(n-5),$ $ L=L_{6,7}^{(2)}(\eta) \oplus A(n-6),$ $ L\cong L_{6,22}(\epsilon) \oplus A(n-6),$ or $L\cong L_1\oplus A(n-7).$
\item[(iv)]If $\dim L^2=2$ and $cl(L)=3,$ then $L\cong L_{4,3}\oplus  A(n-4) $ or  $L\cong L_{5,5}\oplus  A(n-5).$
\end{itemize}
\end{thm}
\begin{proof}
The result follows from  \cite[Theorems 3.3 and 3.6]{ni3}, Theorems \ref{hh} and \ref{2611}.
\end{proof}


{\small\begin{thebibliography}{99}
\bibitem{ba1} P. Batten, K. Moneyhun, E. Stitzinger, On characterizing nilpotent Lie
algebras by their multipliers. Comm. Algebra 24 (1996) 4319-4330.
\bibitem{2b} P. Batten, E. Stitzinger, On covers of Lie algebras, Comm. Algebra 24 (1996) 4301-4317.
\bibitem{3}
F. R. Beyl, U. Felgner, and P. Schmid, On groups occurring as center factor groups, J. Algebra 61 (1970) 161-177.
\bibitem{bos} L. Bosko, On Schur multiplier of Lie algebras and groups of maximal class, Internat. J. Algebra Comput. 20 (2010) 807-821.
\bibitem{cic} S. Cical\`{o}, W. A. de Graaf, C. Schneider, Six-dimensional nilpotent Lie algebras, Linear Algebra Appl. 436 (2012), no. 1, 163-189.



\bibitem{el}
G. Ellis, A non-abelian tensor product of Lie algebras, Glasg. Math. J. 39 (1991) 101-120.
\bibitem{Gr2}M. P.  Gong, Classification  of nilpotent Lie Algebras of dimension $7$ (over Algebraically closed fields and R),
A thesis in Waterloo, Ontario, Canada, 1998.
\bibitem{Gr}W. A. de Graaf, Classification of $6$-dimensional nilpotent Lie algebras over fields of characteristic not $2$,
 Algebra 309 (2007) 640-653.
\bibitem{hall} P. Hall, The classification of prime power groups, J. Reine Angew. Math. 182, (1940)
130-141.
\bibitem{ha}P. Hardy, E. Stitzinger, On characterizing nilpotent Lie algebras by their multipliers t(L) =
3; 4; 5; 6; Comm. Algebra, 1998, 26(11), 3527-3539.
\bibitem{hei}
 H. Heineken, Nilpotent groups of class $2$ that can appear as central quotient groups,
Rend. Sem. Mat. Univ. Padova 84 (1990), 241-248.
\bibitem{pair}F. Johari, M. Parvizi, P. Niroomand, Capability and Schur multiplier of a pair of Lie algebras,
J. Geometry Phys 114 (2017), 184-196.
\bibitem{mon} K. Moneyhun, Isoclinisms in Lie algebras. Algebras Groups Geom. 11 (1994) 9-22.
\bibitem{ni} P. Niroomand, F. G. Russo, A note on the Schur multiplier of a nilpotent Lie algebra, Comm. Algebra 39 (2011) 1293-1297.
\bibitem{ni1} P. Niroomand, F. G. Russo, A restriction on the Schur multiplier of nilpotent Lie algebras, Electron. J. Linear Algebra 22
(2011) 1-9.
\bibitem{ni2} P. Niroomand, On the dimension of the Schur multiplier of nilpotent Lie algebras, Cent. Eur. J. Math. 9 (2011) 57-64.
\bibitem{ni3} P. Niroomand, M. Parvizi, F. G. Russo, Some criteria for detecting capable Lie algebras, J. Algebra 384 (2013) 36-44.
\bibitem{ni4}P. Niroomand, F. Johari, M. Parvizi, On the capability and Schur multiplier of nilpotent Lie
algebra of class two, Proc. Amer. Math. Soc. 144 (2016), 4157-4168.
\bibitem{alam} A. R. Salemkar, V. Alamian and  H. Mohammadzadeh,  Some properties of the Schur multiplier and covers of Lie Algebras, Comm.
Algebra 36 (2008) 697-707.
\bibitem{K}L. M. Zack, Nilpotent Lie algebras with a small second derived quotient, Comm. Algebra, 36 (2008) 460-4619.
\end{thebibliography}}
\end{document}